\documentclass{amsart}
\usepackage{amsmath}
\usepackage{amsfonts}
\usepackage{amssymb}
\usepackage{graphicx}
\usepackage{float, verbatim}
\usepackage{enumerate}
\usepackage{color}
\usepackage{graphicx}
\usepackage{tikz}
\usetikzlibrary{positioning}
\usetikzlibrary{calc}
\usetikzlibrary{arrows,shapes,backgrounds}

\newcommand{\Haus}{\dim_{\mathrm{H}}}
\newcommand{\Assouad}{\dim_\mathrm{A}}
\newcommand{\boxd}{\dim_\mathrm{B}}
\newcommand{\uboxd}{\overline{\dim_\mathrm{B}}}
\newcommand{\lboxd}{\underline{\dim_\mathrm{B}}}

\newtheorem{thm}{Theorem}[section]

\newtheorem{ques}[thm]{Question}
\newtheorem{lma}[thm]{Lemma}

\newtheorem{defn}[thm]{Definition}

\newtheorem{conj}[thm]{Conjecture}

\newtheorem*{thm*}{Theorem}
\newtheorem*{conj*}{Conjecture}
\newtheorem*{lma*}{Lemma}

\begin{document}
	
	\title{Dimensions of triangle sets}
	
	\author{Han Yu}
	\address{Han Yu\\
		School of Mathematics \& Statistics\\University of St Andrews\\ St Andrews\\ KY16 9SS\\ UK \\ }
	\curraddr{}
	\email{hy25@st-andrews.ac.uk}
	\thanks{}

	
	
	\date{}
	
	\dedicatory{}
	
	\begin{abstract}
		In this paper we discuss some dimension results for triangle sets of compact sets in $\mathbb{R}^2$. In particular we prove that for any compact set $F$ in $\mathbb{R}^2$, the triangle set $\Delta(F)$ satisfies
		\[
		\dim_{\mathrm{A}} \Delta(F)\geq \frac{3}{2}\dim_{\mathrm{A}} F.
		\]
		If $\dim_{\mathrm{A}} F>1$ then we have
		\[
		\dim_{\mathrm{A}} \Delta(F)\geq 1+\dim_{\mathrm{A}} F.
		\]
		If $\Assouad F>4/3$ then we have the following  better bound,
		\[
		\Assouad \Delta(F)\geq \min\left\{\frac{5}{2}\Assouad F-1,3\right\}.
		\]
		Moreover, if $F$ satisfies a mild separation condition then the above result holds also for the box dimensions, namely,
		\[
		\underline{\dim_{\mathrm{B}}} F\geq \frac{3}{2}\underline{\dim_{\mathrm{B}}} \Delta(F) \text{ and }\overline{\dim_{\mathrm{B}}} F\geq \frac{3}{2}\overline{\dim_{\mathrm{B}}} \Delta(F).
		\]
	\end{abstract}

	\maketitle
	
	\section{Introduction}
	The main goal of this paper is to count triangles in a given subset of $\mathbb{R}^2$. To be precise, for any set $F\subset\mathbb{R}^2$ we consider the following triangle set
	\[
	\Delta(F)=\left\{r_1,r_2,r_2\in\mathbb{R}_{+}^3: \exists x,y,z\in F, r_1=|x-y|, r_2=|z-y|, r_3=|x-z|  \right\}.
	\] 
	When $F$ is a finite set, $\#\Delta(F)$ is roughly the number of different triangles spanned by $F$, where we used $\#B$ for the cardinality of a set $B$.  Here we consider two triangles to be the same if they are isometric with respect to the Euclidean metric. In fact, $\#\Delta(F)$ counts a particular triangle at least once and at most $6$ times depending on whether the triangle is equilateral, isosceles or scalene.
	
	On one hand, estimating the size of the triangle set $\Delta(F)$ in terms of the size of $F$ can be naturally considered as a ``higher order simplex" version of the distance set problem, as we can regard line segments as $1$-simplices and triangles as $2$-simplices. Some Hausdorff dimension results for triangle sets and more generally $n$-simplex sets can be found in \cite{GI1}, \cite{GI2}, \cite{GILP15}. On the other hand, the study of triangle sets belongs to the more general class of problems which can be best described as ``finding patterns". For example, consider $F\subset\mathbb{R}^2,$ how large must $F$ be to ensure that $F$ contains three vertices of an equilateral triangle? In \cite{M} is was shown that a set of full Hausdorff dimension can avoid containing three vertices of equilateral triangles. However, in \cite[Corollary 1.7]{CLP} it was shown that if $F\subset\mathbb{R}^2$ satisfies some regularity condition concerning Fourier decay then $F$ contains three vertices of an equilateral triangle. For more recent results see \cite{SV} and the references therein.
	
	Now we introduce a bilinear separation condition which is motivated by \cite[Bilinear Distance Conjecture 1.5]{KT}. In what follows we will encounter the notions of dimensions, covering numbers and approximation symbols. They are discussed in detail in Section \ref{Notation}.
	\begin{defn}\label{BiSep}
		For a set $F\subset\mathbb{R}^2$, we say that $F$ is weak-bilinearly separated if:
		\begin{itemize}
			\item[(1):]There exist three disjoint sets $F_1,F_2,F_3$ such that
			\[
			F_1\cup F_2\cup F_3\subset F
			\]
			and such that
			\[
			\min\{\lboxd F_1,\lboxd F_2,\lboxd F_3\}=\lboxd F.
			\]
			\item[(2):]There is a positive number $r>0$ such that any triple $(x,y,z)\in F_1\times F_2\times F_3$ forms a triangle such that
			\[
			|\langle x-y,y-z\rangle|=|x-y||y-z||\sin \theta(x-y,y-z)|>r.
			\] 
			Here $\langle.,.\rangle$ denotes the Euclidean outer product and $\theta(x-y,y-z)$ is the angle spanned by $x-y,y-z$.
		\end{itemize}
		We say that $F$ is strong-bilinarly separated if the condition $(1)$ above is replaced by the following stronger condition.
		\begin{itemize}
			\item[(1'):] There are three disjoint sets $F_1,F_2,F_3$ such that
			\[
			F_1\cup F_2\cup F_3\subset F
			\]
			and such that for all $\delta\in (0,1),$
			\[
			\min\{N_\delta(F_1),N_\delta(F_2),N_\delta(F_3)\}\gtrsim N_\delta(F).
			\]
		\end{itemize}
		
	\end{defn}
	The intuitive idea behind the above bilinear separations is that $F$ can be essentially decomposed into three large parts. If we take one point from each part, the three points form a triangle which is far away from being degenerate. Following this idea, it is natural to guess that if $F$ is not small, then we can find many triangles in $F$. The following theorem formulates this idea more precisely.
	
	\begin{thm}\label{Main}
		Let $F\subset\mathbb{R}^2$ be a compact and weak-bilinearly separated set. Then we have the following inequality
		\[
		\uboxd\Delta(F)\geq \frac{3}{2}\lboxd F.
		\]
		If $F$ is strong-bilinearly separated then we have
		\[
		\lboxd\Delta(F)\geq \frac{3}{2}\lboxd F \text{ and }\uboxd\Delta(F)\geq \frac{3}{2}\uboxd F.
		\]
	\end{thm}
	
	\begin{thm}\label{coMain}
		Let $F\subset\mathbb{R}^2$ be a compact set. Then the following result holds
		\[
		\Assouad \Delta(F)\geq \frac{3}{2}\Assouad F.
		\]
		If $\Assouad F>1$ then we have the following better bound,
		\[
		\Assouad \Delta(F)\geq 1+\Assouad F.
		\]
		If $\Assouad F>4/3$ then we have the following even better bound,
		\[
		\Assouad \Delta(F)\geq \min\left\{\frac{5}{2}\Assouad F-1,3\right\}.
		\]
	\end{thm}
	Note that for the Assouad dimension result we do not need any bilinear separation condition on $F$. In this paper, we will use three different methods to study the triangle sets. For the first part of Theorem \ref{Main} we will use a direct and straightforward counting method, see Lemma \ref{lmDCR} and for Theorem \ref{coMain} we will make use of the fact that there is a close relationship between the triangle sets and the distance sets. For the second part of Theorem \ref{coMain} we shall use a harmonic analysis method introduced in \cite{GILP15}. All these three methods have their advantages. The distance set method is more powerful in $\mathbb{R}^2$ and the direct counting method, as well as the harmonic analysis method, has a potential to be generalized for higher dimensional situations, see Section \ref{Fur}.

	\section{Notations}\label{Notation}
	We briefly introduce the notions of dimensions in this section. For more details on the Hausdorff and box dimensions, see \cite[Chapters 2,3]{Fa} and \cite[Chapters 4,5]{Ma1}. For the Assouad dimension see \cite{F} for more details.
	
	We shall use $N_r(F)$ for the minimal covering number of a set $F$ in $\mathbb{R}^n$ by cubes of side length $r>0$. 
	\subsection{Hausdorff dimension}
	
	For any $s\in\mathbb{R}^+$ and $\delta>0$ define the following quantity
	\[
	\mathcal{H}^s_\delta(F)=\inf\left\{\sum_{i=1}^{\infty}(\mathrm{diam} (U_i))^s: \bigcup_i U_i\supset F, \mathrm{diam}(U_i)<\delta\right\}.
	\]
	The $s$-Hausdorff measure of $F$ is
	\[
	\mathcal{H}^s(F)=\lim_{\delta\to 0} \mathcal{H}^s_{\delta}(F).
	\]
	The Hausdorff dimension of $F$ is
	\[
	\Haus F=\inf\{s\geq 0:\mathcal{H}^s(F)=0\}=\sup\{s\geq 0: \mathcal{H}^s(F)=\infty          \}.
	\]
	\subsection{Box dimensions}
	The upper box dimension of $F$ is
	\[
	\overline{\boxd} F=\limsup_{r\to 0} \left(-\frac{\log N_r(F)}{\log r}\right).
	\]
	The lower box dimension of $F$ is
	\[
	\lboxd F=\liminf_{r\to 0} \left(-\frac{\log N_r(F)}{\log r}\right).
	\]
	If the upper and lower box dimensions of $F$ are equal then we call this value the box dimension of $F$ and we denote it as $\boxd F.$
	\subsection{The Assouad dimension and weak tangents}
	The \textit{Assouad dimension} of $F$ is 
	\begin{align*}
	\Assouad F = \inf \Bigg\{ s \ge 0 \, \, \colon \, (\exists \, C >0)\, (\forall & R>0)\,  (\forall r \in (0,R))\, (\forall x \in F) \\ 
	&N_r(B(x,R) \cap F) \le C \left( \frac{R}{r}\right)^s \Bigg\}
	\end{align*}
	where $B(x,R)$ denotes the closed ball of centre $x$ and radius $R$.
	
	We now introduce the notion of \emph{weak tangent} which is very useful in studying the Assouad dimension. We start by defining the Hausdorff distance between two compact sets $A,B \subset \mathbb{R}^n$ as follows
	\[
	d_{\mathcal{H}}(A,B)=\inf\{\delta>0: A\subset B_{\delta}, B\subset A_{\delta}\},
	\]
	where $A_\delta$ is the closed $\delta$-neighbourhood of $A$. Let $X=[0,1]^n$ for some $n$ then $(\mathcal{K}(X),d_\mathcal{H})$, the space of compact subsets of $X$, is a compact metric space.
	
	\begin{defn}
		We call $D\in \mathcal{K}(X)$ a \emph{miniset} of $F\in\mathcal{K}(R^n)$ if $D = \left(cF + t\right)\cap X$ for some constants $c\ge 1$ and $t\in \mathbb{R}^n$. A set $E \in \mathcal{K}(X)$ is called a \emph{weak tangent} if it is the limit of a sequence of minisets (under the Hausdorff metric).
	\end{defn}
	Due to \cite{Fr,MT} we know that for compact subsets $F\subset \mathbb{R}^n$
	\[
	\Assouad F \ge \sup \left\{ \Assouad E \colon E \textrm{ is a weak tangent of }F \right\}.
	\]
	The following result can be found in \cite[Proposition 5.7]{KOR} the positivity of Hausdorff measure can be found in \cite[Theorem 1.3]{Fr}.
	\begin{lma}[KOR]\label{lm2}
		Let $F$ be a compact set with $\Assouad F=s$. Then there exist a weak tangent $E$ of $F$ such that
		\[
		\Haus E=s,
		\]
		moreover we have
		\[
		\mathcal{H}^s(E)>0.
		\]
	\end{lma}
	In particular, we see that for compact subsets $F\subset \mathbb{R}^n,$
	\[
	\Assouad F = \max \left\{ \Haus E \colon E \textrm{ is a weak tangent of }F \right\}.
	\]
	\subsection{Approximation symbols for box counting}

	When counting covering numbers, it is convenient to introduce notions $\approx, \lesssim, \gtrsim$ for approximately equal, approximately smaller and approximately larger. As our box counting procedure always involves scales, later we use $1>\delta>0$ to denote a particular scale. Then for two quantities $f(\delta), g(\delta)$ we define the following:
	\[
	f\lesssim g\iff \exists M>0, \forall \delta>0, f(\delta)\leq M g(\delta).
	\] 
	\[
	f\gtrsim g\iff g\lesssim f.
	\]
	\[
	f\approx g\iff f\lesssim g \text{ and } g\lesssim f.
	\]
	Thus $f\approx g$ in our box counting procedure can be intuitively read as ``$f$ and $g$ give the same box dimension". Later in context, $f,g$ can be either box covering number of scale $\delta$ or Lebesgue measure of a $\delta$ neighbourhood of a set.
	
	\section{Some recent results}
	In this section we review some results on triangle sets $\Delta(F)$ for finite $F\subset\mathbb{R}^2$. A result by \cite{GK} shows that there exists a constant $C>0$ such that
	\[
	\#D(F)\geq C \frac{\#F}{\log \#F},
	\]
	where we write $D(F)$ for the distance set of $F$ defined as follows,
	\[
	D(F)=\{|x-y|: (x,y)\in F\times F\}.
	\]
	From here (see \cite[section 6]{GI1} and the references therein) it is possible to show that
	\[
	\#\Delta(F)\geq C \frac{\#F^2}{\log \#F}.
	\] 
	
	For Borel sets, by a result due to \cite[Theorem 1.5]{GILP15} it is possible to deduce the following result.
	\begin{thm}[GILP]\label{GIth}
		Let $F\subset\mathbb{R}^2$ be a Borel set. If $\Haus F>\frac{8}{5}$ then $\Delta(F)$ has positive Lebesgue measure.
	\end{thm}
	
	We note here a recent result on the distance sets that will be used later in this paper. Notice that the theorem we stated below is weaker than its original version \cite[Theorem 1.1]{S}. In what follows for $x\in\mathbb{R}^2$ we denote the pinned distance set $D_x(F)$ to be the following set,
	\[
	D_x(F)=\{|x-y|: y\in F\}.
	\]
	\begin{thm}\label{S}
		Let $F\subset\mathbb{R}^2$ be a Borel set with $\Haus F=\uboxd F>1$ then we have the following result,
		\[
		\Haus \{x\in\mathbb{R}^2:    \Haus D_x(F)<1\}\leq 1.
		\]
	\end{thm}
	\section{some combinatorics, following \cite{KY}}
	To prove Theorem \ref{Main} we need some combinatorial results. In this section, we shall prove Lemma \ref{lmDCR}. This result is based on a private communication with K. H\'{e}ra \cite{KY}. In order to improve readability and to make our paper self-contained, we provided detailed proofs here. We begin with two standard combinatorial tools.
	\begin{lma}[Chessboard]\label{chessboard}
		Let $n\geq 1$ be an integer. Let $\delta>0$ and consider a disjoint covering of $[0,1]^n$ with $\delta$-cubes. For any collection of $N\geq 100^n$ such cubes, we can find a sub-collection with at least $\frac{N}{100^n}$ cubes such that each two different cubes in this sub-collection are separated by at least $100\delta$. 
	\end{lma}
	\begin{proof}
		The idea can be presented in the following picture in $\mathbb{R}^2.$
		
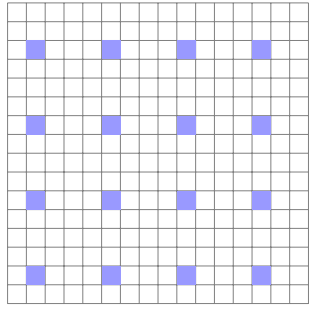
\begin{center}
			\begin{tikzpicture}
			\draw[step=0.25cm, gray,very thin] (-2,-2) grid (2,2);
			
			\foreach \x in {0,...,3}{
				\foreach \y in {0,...,3} {
					\fill[blue!40!white] (-2+\x+0.5,-2+\y+0.5) rectangle (-2+\x+0.25,-2+\y+0.25);}
			}
			\end{tikzpicture}
		\end{center}
		We can assume $\delta=1/K$ for some integer $K$. Then our $\delta$-cubes can be identified with a subset of $\mathbb{N}^n$. Namely
		\[
		A=\{(a_1,a_2,\dots,a_n)\in\mathbb{N}^n:\forall i\in [1,n], a_i\in [1,K]\}.
		\]
		Let $i\in [1,n], r_i\in [0,99]$ be $n$ integers. Consider the set
		\[
		A_{r_1,\dots,r_n}=\{(a_1,a_2,\dots,a_n)\in A: i\in [1,n],a_i\equiv r_i \mod 100\}.
		\]
		There are $100^n$ many such sets, they are pairwise disjoint and they form a decomposition of $A$. Therefore for any subset of $A$ with $N$ elements, we can find integers $i\in [1.n], r_i$ such that
		$
		A_{r_1,\dots,r_n}
		$
		contains at least $N/100^n$ many elements. By construction, two different points in $    A_{r_1,\dots,r_n}$ are at least $100$ separated. The proof concludes by scaling the whole configuration back by multiplying $\delta=K^{-1}$.
	\end{proof}
	\begin{lma}[Dyadic pigeonhole]\label{DyPi}
		Let $N$ be an integer, suppose there are $N$ positive integers $m_i,i\in\{1,\dots,N\}$ such that $\sum_{i} m_i=M$. Then we can find a positive integer $k$ such that $\sum_{i:m_i\in [2^{k-1},2^{k})}m_i\geq \frac{M}{\log M}$.
	\end{lma}
	\begin{proof}
		For each $k\in\mathbb{N}$ consider
		\[
		D_k=\{i\in\{1,\dots,N\}: m_i\in [2^{k-1},2^{k})\}.
		\]
		It is clear that for each $i$ we have $1\leq m_i\leq M$ and therefore we have at most $\log M$ many non empty sets $D_k$. Therefore there is at least one $k$ such that
		\[
		\sum_{i\in D_k}m_i\geq \frac{M}{\log M}.
		\]    
	\end{proof}
	We will now prove the following combinatorial result.
	\begin{lma}[KY]\label{lmDCR}
		Let $0<r_1\leq r_2$ be positive real numbers. Let $\gamma>0$ be another positive number such that $0<10\gamma<r_2+r_1$. Let $\delta>0$ be a positive real number which can be chosen arbitrarily small. Let $\alpha,\beta$ be two non negative numbers such that we can find a set $C\subset\mathbb{R}^2$ such that:
		\begin{itemize}
			\item{1}: $C$ is $100\delta$ separated and $\#C=\delta^{-\alpha}$.
			\item{2}: For any $c\in C$, two circles $c+r_1S^1$ and $c+r_2S^1$ both contain $\delta^{-\beta}$ many $100\delta$ separated points. We denote the two sets of points as $C_1(c)$ and $C_2(c)$ respectively.
			\item{3}: For any $c\in C$, we have the following condition
			\[
			\#\left\{(x,y)\in C_1(c)\times C_2(c): r_2-r_1+10\gamma<|x-y|<r_2+r_1-10\gamma             \right\}\geq \frac{1}{2} \#C_1(c)\#C_2(c).
			\]
		\end{itemize}
		
		Then $\bigcup_{c\in C} (C_1(c)\cup C_2(c))$ contains $\gtrsim \delta^{-0.5\alpha-\beta}$ many $\delta$-separated points. The implicit constant in $\gtrsim$ depends on $r_1,r_2,\gamma$ but not on $\alpha, \beta$.
	\end{lma}
	
	\begin{proof}
		We can create a $\delta$-grid and count how many disjoint $\delta$-squares we need to cover $\bigcup_{c\in C}C_1(c)\cup C_2(c)$. We have in total $ 2\delta^{-\alpha-\beta}$ many possible points, but they might not be all $\delta$-separated. We denote $\mathcal{F}$ the set of $\delta$-squares we need to cover all points. Consider the following incidence set
		
		\begin{eqnarray*}
			\mathcal{I}&=&\big\{(x_1,x_2,c)\in F\times F\times C: \\
			& &x_1\cap c+r_1S^1 \neq\emptyset, x_2\cap c+r_2S^1\neq\emptyset, |x_1-x_2|\subset \left(r_2-r_1+\gamma,r_2+r_1-\gamma\right)\big\}.
		\end{eqnarray*}
		
		We now count $\#\mathcal{I}$ in two different ways. First, there are at most $\#\mathcal{F}^2$ many elements in $\mathcal{F}\times \mathcal{F}$. For any pair $(x_1,x_2)\in \mathcal{F}\times \mathcal{F}$ such that $|x_1-x_2|\subset (r_2-r_1+\gamma,r_2+r_1-\gamma)$, it is not hard to see that there are not many $c\in C$ such that
		\[
		x_1\cap c+r_1S^1 \neq\emptyset, x_2\cap c+r_2S^1\neq\emptyset.
		\]
		To be precise, consider $T_{r_1}^{3\delta}(x_1), T_{r_2}^{3\delta}(x_2)$ to be the annulus cocentred with $x_1, x_2$ respectively. The inner and outer radii of $T_{r_i}^{3\delta}(x_i)$ are $r_i-1.5\delta, r_i+1.5\delta$. Then we see that
		\[
		\{c\in C: x_1\cap c+r_1S^1 \neq\emptyset, x_2\cap c+r_2S^1\neq\emptyset\}\subset T_{r_1}^{3\delta}(x_1)\cap T_{r_2}^{3\delta}(x_2).
		\]
		Now observe that there is a constant $A$ which depends on $r_1,r_2,\gamma$ such that
		\[
		T_{r_1}^{3\delta}(x_1)\cap T_{r_2}^{3\delta}(x_2)
		\]
		can be covered by two squares side length $A\delta$. Because $C$ is $\delta$-separated we see that
		\[
		\{c\in C: x_1\cap c+r_1S^1 \neq\emptyset, x_2\cap c+r_2S^1\neq\emptyset\}
		\]
		contains at most $1000A^2$ many elements. Therefore we see that
		\[
		\#\mathcal{I}\leq 1000A^2\#\mathcal{F}^2.\tag{1}
		\]
		To obtain a lower bound, we consider each individual $c\in C$. By assumption for each $c\in C$ we can find at least $0.5 \#C_1(c)\#C_2(c)$ many pairs in $\mathcal{F}\times \mathcal{F}$ such that the distance set between each pair of cubes are contained in $(r_2-r_1+\gamma,r_2-r_1-\gamma)$. Therefore we see that
		\[
		\#\mathcal{I}\geq \frac{1}{2} \delta^{-\alpha-2\beta}.\tag{2}
		\]
		The final result follows from $(1),(2)$ by absorbing constants into $\gtrsim$ symbol.
	\end{proof}
	
	\section{proof of Theorem \ref{Main}}
	In the proof we will be counting covering numbers with a fixed scale $\delta\in (0,1)$. All the quantities might depend on $\delta$. We will also use $\gtrsim, \lesssim, \approx$ symbols. In this proof, we can see that there is a fixed number $M>0$ such that all the implicit constants in the approximation symbols can be chosen as $M$. For example
	\[
	f(\delta)\lesssim g(\delta) \text{ can be written as } f(\delta)\leq M g(\delta).
	\]
	Let $F_1, F_2, F_3$ be stated as in the statement of Theorem \ref{Main}. We first show the result with weak-bilinear separation condition. For all $\delta\in (0,1)$ and $\epsilon>0$  we have the following result
	\[
	\min\{N_\delta(F_1),N_\delta(F_2),N_\delta(F_3)\}\gtrsim \delta^{-\lboxd F+\epsilon}.\tag{*}
	\]
	Later we shall let $\epsilon\to 0$ but for now it is a fixed number. By applying the chessboard argument (Lemma \ref{chessboard}) we can find $100\delta$-separated subsets $C_i\subset F_i, i\in \{1,2,3\}$ such that
	\[
	\#C_i\gtrsim  \delta^{-\lboxd F+\epsilon}.
	\]
	We see that each $(c_1,c_2,c_3)\in C_1\times C_2\times C_3$ forms a triangle which is far away from being degenerate. (A precise description can be found in Definition \ref{BiSep}.) Now we can use disjoint $\delta$ cubes to cover $\Delta(F)$. Denote this set of $\delta$-cubes to be $\mathcal{N}_{\delta}(\Delta(F))$. We see that the triple $(|c_1-c_2|,|c_2-c_3|,|c_3-c_1|)$ belongs to one of the cubes in $\mathcal{N}_{\delta}(\Delta(F))$. We denote $K(c_1,c_2,c_3)$ to be this cube. For each $K\in \mathcal{N}_{\delta}(\Delta(F))$ we define the following set
	\[
	S(K)=\{(c_1,c_2,c_3)\in C_1\times C_2\times C_3:(|c_1-c_2|,|c_2-c_3|,|c_3-c_1|)\in K \}.
	\]
	Clearly we have the following relation
	\[
	\sum_{K\in \mathcal{N}_{\delta}(\Delta(F))} \#S(K)= \#C_1\#C_2\#C_3.
	\]
	Then by pigeonhole principle we see that there exists at least one $K$ such that
	\[
	\#S(K)\geq \frac{\#C_1\#C_2\#C_3}{N_\delta(\Delta(F))}\gtrsim \frac{ (\delta^{-\lboxd F+\epsilon})^3}{N_\delta(\Delta(F))}.
	\]
	Intuitively this means that we can find a triangle and we can find a lot of copies of that triangle in $F$. Now we fix this choice of cube $K$. 
	For each $c_1\in C_1$, we define the following set
	\[
	S(K,c_1)=\{(c_2,c_3)\in C_2\times C_3: (c_1,c_2,c_3)\in S(K)\}.
	\]
	Then we see that
	\[
	\sum_{c_1\in C_1} \#S(K,c_1)=\#S(K).
	\]
	By dyadic pigeonhole principle (Lemma \ref{DyPi}) we can find two integers $N_1, N_2$ such that
	\[
	N_1N_2\geq \frac{\#S(K)}{2\log \#S(K)}
	\]
	and such that there are $N_1$ many $c_1\in C_1$ with
	\[
	\#S(K,c_1)\in [N_2, 2N_2].
	\]
	Now we take a closer look at the set $S(K,c_1)$ for $c_1$ described as above. We can find many pairs $(c_1,c_2,c_3)$ in $C_2\times C_3$ such that $c_1,c_2,c_3$ is $\delta$-close to a fixed non degenerate triangle. Because of the bilinear separation condition, there exist positive constants $r_1,r_2,\gamma$ such that 
	\[
	|c_1-c_2|\in [r_1-3\delta, r_1+3\delta], |c_1-c_3|\in [r_2-3\delta, r_2+3\delta], |c_2-c_3|\in [|r_1-r_2|+10\gamma, r_1+r_2-10\gamma].
	\]  
	If we further fix $c_2$, then there exist at most a bounded number (which does not depend on $\delta$) of $c_3$ such that
	\[
	(c_2,c_3)\in S(K,c_1). 
	\]
	To summarize, around each $c_1$ we can find two (not necessary distinct) annulus of inner,outer radii in $[r_1-3\delta, r_1+3\delta], [r_2-3\delta, r_2+3\delta]$. Those annuli contains $\gtrsim \#S(K,c_1)$ many points in $C_2, C_3$ respectively. Now we can apply Lemma \ref{lmDCR} to deduce that
	\[
	N_\delta(F)\gtrsim \sqrt{N_1N^2_2}.
	\]
	We also have the following obvious bound
	\[
	N_{\delta}(F)\geq \max\{N_1, N_2\}.
	\]
	So we see that
	\[
	N_\delta(F)\gtrsim (N_1N_2)^{\frac{2}{3}}\gtrsim \frac{-1}{\log \delta} \frac{(\delta^{-\lboxd F+\epsilon})^2}{N_\delta(\Delta(F))^{2/3}}.
	\]
	Therefore we see that for all $\delta>0$
	\[
	N_\delta(\Delta(F))\gtrsim \left(\frac{-1}{\log \delta}\right)^{3/2} \frac{(\delta^{-\lboxd F+\epsilon})^3}{(N_\delta(F))^{3/2}}.
	\]
	We know that there exist arbitrarily small $\delta\in (0,1)$ such that
	\[
	N_\delta(F)\leq \delta^{-\lboxd F-\epsilon}.
	\]
	This implies that
	\[
	\uboxd \Delta(F)\geq \frac{3}{2}\lboxd F-4.5\epsilon.
	\]
	By letting $\epsilon\to 0$ we see that
	\[
	\uboxd \Delta(F)\geq \frac{3}{2} \lboxd F.
	\]
	Now we shall show the result for strong-bilinear separation condition. In this case we can replace the inequality $(*)$ with the following,
	\[
	\min\{N_\delta(F_1),N_\delta(F_2),N_\delta(F_3)\}\gtrsim N_\delta(F).
	\]
	The rest of this proof is very similar to that of the weak-bilinear separation case. We omit the full details.
	\section{Proof of Theorem \ref{coMain}, part I}
	Before we prove Theorem \ref{coMain}, let us examine an extreme case when the weak-bilinear separation condition does not hold. Suppose that $F$ is contained in a line segment and for simplicity we shall assume that $F\subset [0,1]$.  For any $\delta\in (0,1)$ we can cover $F$ with disjoint $\delta$-boxes and we need $N_{\delta}(F)$ many of them. Then we simply find $x\in F$ such that \[N_{\delta}(F\cap [0,x])=N_{\delta}(F\cap [x,1])\] and therefore for any $y,z\in F$ with $y<x<z$ we see that $(x-y,z-x)\in \Delta(F)$. It is easy to see that this gives at least $0.25 N_\delta(F)^2$ many contributions to $N_\delta(\Delta(F))$. As this holds for all $\delta\in (0,1)$ we see that $\lboxd \Delta(F)\geq 2\lboxd F.$ We will use this result later. 
	
	\begin{lma}\label{lm3}
		If $F\subset \mathbb{R}^2$ is a compact subset then $\Delta(F)$ is a compact subset of $\mathbb{R}^3.$
	\end{lma}
	\begin{proof}
		It is clear that $\Delta(F)\subset\mathbb{R}^3$ is bounded. We show that $\Delta(F)$ is also closed. Let $a_i\in \Delta(F),i\geq 1$ be a sequence of points converging to $a\in\mathbb{R}^3.$ Then we can find points $(x_i,y_i,z_i)\in F\times F\times F$ such that
		\[
		a_i=(|x_i-y_i|,|x_i-z_i|,|y_i-z_i|).
		\]
		By taking a subsequence if necessary we assume that $x_i\to x, y_i\to y, z_i\to z$ for $(x,y,z)\in F\times F\times F.$ Then we see that
		\[
		|x_i-y_i|\to |x-y|,|x_i-z_i|\to |x-z|,|z_i-y_i|\to |z-y|.
		\]
		Thus we see that $a\in\Delta(F).$ This concludes the proof.
	\end{proof}
	\begin{lma}\label{lm1}
		Let $E$ be a weak tangent of $F$, then $\Delta(E)$ is a subset of a weak tangent of $\Delta(F)$. In particular we see that $\Assouad \Delta(E)\leq \Assouad \Delta(F).$
	\end{lma} 
	\begin{proof}
		By definition, $E$ is the limit in Hausdorff metric of sets $E_i=(r_i E+b_i )\cap [0,1]^2$ where
		\[
		b_i\in\mathbb{R}^2, r_i>0, \lim_{i\to\infty} r_i=\infty.
		\]
		Now we can take the sequence $\Delta(E_i)$ and it is easy to see that
		\[
		\Delta(E_i)\subset r_i\Delta(F)\cap [0,1]^3.
		\]
		By taking a subsequence of the sets $E_i$ if necessary we can assume that
		\[
		E_i, \Delta(E_i),  r_i\Delta(F)\cap [0,1]^3
		\]
		all converge as $i\to\infty$ with respect to the Hausdorff metric. Now fix a positive number $\epsilon>0$ and suppose that $(a,b,c)\in \Delta(E)$. Then for all large enough $i$ we can find three points $(a_i,b_i,c_i)$ in $E_i$ such that
		\[
		\max\{|a_i-a|, |b_i-b|, |c_i-c|\}<\epsilon.
		\]
		Then we see that
		\[
		|a_i-b_i|\leq |a_i-a|+|a-b|+|b-b_i|<|a-b|+2\epsilon
		\]
		\[
		|a-b|\leq |a-a_i|+|a_i-b_i|+|b_i-b|<|a_i-b_i|+2\epsilon.
		\]
		Similar relations hold for $|b-c|, |b_i-c_i|$ and $|a-c|,|a_i-c_i|$ as well. We see that 
		\[
		\Delta(E)\subset\lim_{i\to\infty} \Delta(E_i)\subset \lim_{i\to\infty}(r_i\Delta(F)\cap [0,1]^3).
		\]
		This is what we want.
	\end{proof}
	
	Because of Lemma \ref{lm2}, Lemma \ref{lm3} and Lemma \ref{lm1}, by taking a weak tangent if necessary, we can assume that $F$ has equal Hausdorff dimension and Assouad dimension, say,  $s>0$ and $\mathcal{H}^s(F)>0.$ We shall try to find an integer $k>0$ and three disjoint dyadic cubes
	\[
	c_1,c_2,c_3\in \mathcal{D}_k
	\]
	such that for $i\in\{1,2,3\}$
	\[
	\lboxd (c_i\cap F)= s.
	\]
	In what follows for each $i$ we write $F_i$ as $c_i\cap F$. Furthermore any triple $(x,y,z)\in F_1\times F_2\times F_3$ form a triangle that is far away from being degenerate, namely, there is a constant $r>0$ such that for all such triples
	\[
	|\langle x-y, y-z\rangle|>r.
	\] 
	If we can find such dyadic cubes then $F_1\cup F_2\cup F_3$ is weak-bilinearly separated and by Theorem \ref{Main} we see that
	\[
	\Assouad \Delta(F)\geq\uboxd \Delta(F)\geq \uboxd \Delta(F_1\cup F_2\cup F_3)\geq \frac{3}{2}\lboxd (F_1\cup F_2 \cup F_3)\geq \frac{3}{2}s.
	\]
	
	Now we are going to find those dyadic cubes. We shall see that it is always possible to find such cubes unless $F$ is essentially contained in a line in a precise sense. First we want to find $k_1$ such that at least two non-adjacent cubes in $\mathcal{D}_{k_1}$ whose intersections with $F$ have positive $\mathcal{H}^s$ measures. If such $k_1$ does not exist, then we see that $\mathcal{H}^s$ has singleton support and this is not possible. This contradiction gives us the existence of $k_1$. Then we can find  two non-adjacent cubes $A_1,A_2$ whose intersections with $F $ has positive $\mathcal{H}^s$ measures. Then for any $k>k_1$ we can find $A^{k}_1,A^{k}_2\in\mathcal{D}_k$ with $A^{k}_1\subset A_1, A^{k}_2\subset A_2$. Then we define the following ``line" set
	\[
	L(A^{k}_1,A^{k}_2)=\bigcup^{l\cap A^k_1\neq\emptyset,l\cap A^k_2\neq\emptyset}_{l \text{ is a line}}l\cap F.
	\]
	Since $F$ is bounded, we see that as $k\to\infty$, $L(A^{k}_1,A^{k}_2)$ converges to a line segment $L$ with respect to the Hausdorff metric. Now if for any $k>k_1$ we can find $A^{k}_3\in\mathcal{D}_k$ whose intersection with $F$ has positive $\mathcal{H}^s$ measure such that 
	\[
	2A^{k}_3\cap L(A^{k}_1,A^{k}_2)=\emptyset,
	\]
	then we can choose $c_i=A^{k}_i$ for $i\in\{1,2,3\}$ and we are done. Otherwise we see that $\mathcal{H}^s$ is supported in a $3\times 2^{-k}$ neighbourhood of $L$ for all $k>k_1$ and therefore $\mathcal{H}^s$ is supported in $L$. Therefore we can actually focus on $F\cap L$ and in this case we saw that $s\leq 1$ and \[\Assouad \Delta(F)\geq \lboxd \Delta(F)\geq 2 \lboxd F\geq 2\Haus F=2\Assouad F.\] 
	In the right most inequality we have used the assumption that $\Haus F=\Assouad F.$ This shows that
	\[
	\Assouad \Delta(F)\geq \frac{3}{2}\Assouad F.\tag{\%}
	\]
	Now we assume that $\Assouad F=s>1$ and in this case because of Lemma \ref{lm2}, Lemma \ref{lm3} and Lemma \ref{lm1} as before we can assume that $\Haus F=\Assouad F=s>1$ and $\mathcal{H}^s(F)>0$. Then we see that $F$ is weakly-bilinearly separated in this case because $s>1.$ As above, we can find subsets $F_1,F_2,F_3\subset F$ with positive $\mathcal{H}^s$ measures. Since $s>1$ and $\Haus F_1\leq \Assouad F=s$ we see that $\Haus F_1=\Assouad F_1$ and the same relation holds for $F_2,F_3$ as well. By Theorem \ref{S} we see that there exists $x\in F_1$ such that $\Haus D_x(F_2)=1.$ Then we see that
	\[
	\lboxd D_x(F_2)=1.
	\]
	For each $\epsilon>0,$ for all small enough number $\delta>0$ we see that \[N_\delta(D_x(F_2))\geq \delta^{-1+\epsilon}.\tag{\#}\] 
	Now we choose $l\in D_x(F_2)$ and $y\in F_2$ such that $|x-y|=l.$ Because of the construction of $F_1,F_2,F_3$ we can assume that $l\geq c$ for a constant $c>0$ which does not depend on $\delta.$ Therefore we see that there exists constant $M>0$ such that for each pair of two points $z,z'\in F_3$ with $|z-z'|\geq M\delta$, the triangle spanned by $xyz$ and the triangle spanned by $xyz'$ separate each other by at least $\delta$ when regarded as points in $\mathbb{R}^3.$ Since $\Haus F_3=\Assouad F_3=s$ we see that $\lboxd F_3=s$ and for all small enough $\delta$ we have $N_\delta(F_3)\geq \delta^{-s+\epsilon}.$ Then together with $(\#)$ we see that for all small enough $\delta>0,$
	\[
	N_\delta(\Delta(F_1\cup F_2\cup F_3))\geq \delta^{-1-s+2\epsilon}.
	\]
	This implies that $\lboxd \Delta(F)\geq 1+s-2\epsilon.$ Since $\epsilon>0$ can be chosen arbitrarily small we see that
	\[
	\Assouad \Delta(F)\geq \lboxd \Delta(F)\geq 1+s.
	\]
	The above result holds for $\Assouad F=s>1$ and together with $(\%)$ we see that the first two conclusions of Theorem \ref{coMain} concludes. 
	\section{Proof of Theorem \ref{coMain}, part II}\label{HAR}
	In this section, we closely follow \cite{GILP15}. At this state, we are not aiming at self-containing. The reader is strongly recommended to read \cite{GILP15} and convince himself/herself that the result we are going to prove `naturally follows'. In fact, a fairly large part of the main proof in this section shares arguments with \cite{GILP15}.
	
	As we use some different notations than in \cite{GILP15}, we reintroduce some definitions in \cite{GILP15}.
	
	\begin{defn}
		Let $n\geq 2$ and $2\leq k\leq n+1$ be integers. Given a set $F\subset\mathbb{R}^n$, define 
		\[
		\Delta_k(F)=\{(r_{ij},1\leq i<j\leq k)\in\mathbb{R}^{k(k-1)/2}: x_1,\dots,x_k\in F, |x_i-x_j|=r_{ij}, 1\leq i<j\leq k  \}.
		\]
	\end{defn}
	Notice that $\Delta_3$ has the same meaning as $\Delta$ we have dealt with. Taking permutations of vertices into account, $\Delta_k(F)$ counts a particular simplex at least once and at most $c(k)$ times for an integer $c(k)$ depending only on $k.$ For $k=3$ we know that $c(3)=6.$ 
	\begin{defn}\label{Def1}
		Let $F\subset\mathbb{R}^n$ be a compact set and let $\mu$ be a probability measure supported on $F$. For $g\in\mathbb{O}(n)$, the orthogonal group on $\mathbb{R}^n$, we construct a measure $\nu_g$ as follows,
		\[
		\int_{\mathbb{R}^n} f(z)d\nu_g(z)=\int_F\int_F f(u-gv)d\mu(u)d\mu(v),f\in C_0(\mathbb{R}^n).
		\]
		In other words, $\nu_g=\mu*g\mu,$ where $g\mu$ is the pushed forward measure of $\mu$ under the map $g.$ We also construct a measure $\nu$ on $\Delta_k(F)\subset\mathbb{R}^{k(k-1)/2}$ by
		\[
		\int f(t)d\nu(t)=\int f(|x_1-x_2|,\dots,|x_i-x_j|,\dots,|x_{k-1}-x_k|)d\mu(x_1)\dots d\mu(x_k), f\in C_0(\mathbb{R}^{k(k-1)/2}),
		\]
		where $t$ is a $k(k-1)/2$-vector with entries $|x_i-x_j|$ for $1\leq i<j\leq k.$
	\end{defn}
	For a given set $F\subset\mathbb{R}^n$ we can choose $\mu$ supported on $F$ with some regularities, for example, a Frostman measure. Then with these regularities we are able to obtain some results of $\nu.$ Since $\nu$ supports on $\Delta_k(F)$ we can get some informations for $\Delta_k(F)$. Thus the difference between $\Delta_k(F)$ in the above definition and $T_k(E)$ in \cite[Definition 1,1]{GILP15} is that we count the same simplex multiple times due to permutations of its vertices. For example, a triangle $\Delta ABC$ would count differently than $\Delta BAC$ in our triangle set counting, but they are actually the same triangle. A bit of caution should be given here. We frequently use $B_\delta(x)$ for the (closed) $\delta$-ball around $x$ in a metric space. We will encounter the situation where we uses $B_{\delta}(x), B_\delta(z)$ in the same expression but $x,z$ are in different spaces. We hope no confusion will rise here as the closed balls should be in the same space as their centres.
	
	\begin{thm}\label{HAR1}
		Let $F\subset\mathbb{R}^n$ be a compact set with $\Haus F=s>n/2.$ Then we have
		\[
		\Haus \Delta(F)\geq \min\{2s+\gamma_s-2n+3,3\},
		\]
		where $\gamma_s=(n+2s-2)/4$ if $s\in [n/2,(n+2)/2]$ and $\gamma_s=s-1$ if $s\geq (n+2)/2.$ 
	\end{thm}
	In particular, when $n=2$ we see that
	\[
	\Haus \Delta(F)\geq 2.5s-1.
	\]
	Then the third conclusion of Theorem \ref{coMain} follows by the above Theorem and the weak tangent trick introduced in the previous section. When $s$ is large enough so that $2s+\gamma_s-2n+3>3$ then $\Delta(F)$ has positive measure, this was shown in \cite{GILP15}.
	\begin{proof}[Proof of Theorem \ref{HAR1}]
		We need to choose cutoff functions in various places. Unless otherwise mentioned, all cutoff functions are assumed to be real valued and radial. Let $\phi(.)$ be a function taking arguments in $\mathbb{R}^l$ with $l\geq 1,l\in\mathbb{N}.$ We say that $\phi$ is radial if $\phi(x)=\phi(y)$ whenever $|x|=|y|.$ A particular reason for requiring the radial property is that Fourier transforms of a real valued radial functions are also real valued and radial. Let $\mu$ be an $s$-Frostman measure supported on $F$, $s$ can be arbitrarily chosen as long as $s<\Haus F.$ Then we construct measures $\nu_g,g\in\mathbb{O}(n), \nu$ as in Definition \ref{Def1}. To start with,  we choose a Schwartz function $\phi\in \mathcal{S}(\mathbb{R}^{k(k-1)/2})$ bounded by $1.5$ whose support is contained in $B_1(0)$ and equal to $1.5$ on $B_{1/2}(0)$. We also require that $\|\phi\|_1=1.$ For any positive number $\delta>0,$ let $\phi_\delta(.)=\delta^{-k(k-1)/2}\phi(./\delta).$ Let $\nu_\delta=\phi_\delta*\nu$. We know that $\nu_\delta\to \nu$ as $\delta\to 0$ in the weak-* sense. By the argument in \cite[Section 2]{GILP15} we see that
		\begin{eqnarray*}
			& &\int \nu^2_\delta(z)dz\lesssim \\
			& &\delta^{-n(k-1)} \int\mu^{2k}\{(x_1,\dots,x_k,y_1,\dots,y_k)\in\mathbb{R}^{2kn}: |(x_i-gy_i)-(x_j-gy_j)|\leq \delta,1\leq i<j\leq k\}dg,
		\end{eqnarray*}
		where $dg$ is the Haar measure on $\mathbb{O}(n).$
		Now we claim that for each $g\in\mathbb{O}(n),$
		\[
		\mu^{2k}\{(x_1,\dots,x_k,y_1,\dots,y_k)\in\mathbb{R}^{2kn}: |(x_i-gy_i)-(x_j-gy_j)|\leq \delta,1\leq i<j\leq k\}\leq \int \nu^{k-1}_g(B_{2\delta}(z))d\nu_g(z).
		\]
		To see this, let $x_1,\dots,x_{k-1},y_1,\dots,y_{k-1}$ be given, consider the following section (we omit the coordinates $x_1,\dots,x_{k-1},y_1,\dots,y_{k-1}$ as they are fixed)
		\[
		\{(x_k,y_k): |(x_i-gy_i)-(x_j-gy_j)|\leq \delta,1\leq i<j\leq k\}
		\]
		it is easy to see that the above section is contained in
		\[
		E=\{(x_k,y_k): |(x_k-gy_k)-(x_1-gy_1)|\leq \delta\}.
		\]
		Then we see that the $\mu^{2k}$ measure is now bounded from above by
		\[
		\mu^{2(k-1)}\{(x_1,\dots,x_{k-1},y_1,\dots,y_{k-1})\in\mathbb{R}^{2(k-1)n}: |(x_i-gy_i)-(x_j-gy_j)|\leq \delta,1\leq i<j\leq k-1\}\times
		\]
		\[
		\int 1_E(x_k,y_k)d\mu(x_k)d\mu(y_k).
		\]
		We see that $1_E(x_k,y_k)=f(x_k-gy_k)$ for $f:z\in\mathbb{R}^n\to f(z)=1_{\{a:|a-(x_1-gy_1)|\leq \delta\}}(z).$ Then by the definition of $\nu_g$ we see that
		\[
		\int 1_E(x_k,y_k)d\mu(x_k)d\mu(y_k)\leq\nu_g(B_{2\delta}(x_1-gy_1)).
		\]
		In fact if $\nu_g$ does not give positive measure on any spheres then $\nu_g(B_{2.5\delta}(.))$ would be continuous with compact support and we would get
		\[
	\int 1_E(x_k,y_k)d\mu(x_k)d\mu(y_k)=\nu_g(B_{\delta}(x_1-gy_1)).
	\]
		However, we do not assume this continuity of $\nu_g$ and we only have an upper bound by choosing a real valued function in $C_0(\mathbb{R}^n)$ which is bounded from above by one, equal to one on $B_{\delta}(x_1-gy_1)$ and vanishes outside $B_{2\delta}(x_1-gy_1).$ We can do the above step $k-1$ times and as a result we see that for each fixed $x_1,y_1,$ the section (we omit the coordinates $x_1,y_1$ as they are fixed)
		\[
		\{(x_2,\dots,x_k,y_2,\dots,y_k)\in\mathbb{R}^{2kn}: |(x_i-gy_i)-(x_j-gy_j)|\leq \delta,1\leq i<j\leq k\}
		\]
		has $\mu^{2k-2}$ measure at most
		\[
		\nu^{k-1}_g(B_{2\delta}(x_1-gy_1)).
		\]
		By Fubini, we see that 
		\[
		\mu^{2k}\{(x_1,\dots,x_k,y_1,\dots,y_k)\in\mathbb{R}^{2kn}: |(x_i-gy_i)-(x_j-gy_j)|\leq \delta,1\leq i<j\leq k\}
		\]
		\[
		\leq \int \nu^{k-1}_g(B_{2\delta}(x_1-gy_1))d\mu(x_1)d\mu(y_1)\leq\int \nu^{k-1}_g(B_{2.5\delta}(z))d\nu_g(z).
		\]
		If $\nu_g(B_{2\delta}(.))$ would be continuous then we would have
		\[ \int \nu^{k-1}_g(B_{2\delta}(x_1-gy_1))d\mu(x_1)d\mu(y_1)=\int \nu^{k-1}_g(B_{2\delta}(z))d\nu_g(z).
		\]
		In general, we choose a continuous function sandwiched by $\nu_g(B_{2\delta}(.))$ and $\nu_g(B_{2.5\delta}(.))$ (by taking a convolution of a suitable smooth cutoff function with $\nu_g$) then apply the definition of $\nu_g$ to arrive at the above inequality. Thus we have shown the following estimate,
		\[
		\int \nu^2_\delta(z)dz\lesssim \delta^{-n(k-1)}\int\int \nu^{k-1}_g(B_{2.5\delta}(z))d\nu_g(z)dg.
		\]
	If $\nu_g$ would be absolutely continuous with respect to the Lebesgue measure for almost all $g\in\mathbb{O}(n),$ then the RHS above could be replaced by
	\[
	\int\int \nu^{k}_g(z)dzdg.
	\]
	Then we would arrive at the same situation as in \cite{GILP15}. In general, we do not have this continuity at hand. To deal with this issue, let $\phi^{DD}(.)$ be a radial Schwartz function such that $\hat{\phi}^{DD}$ is non-negative,  vanishes outside the ball of radius $0.5c''>0$ around the origin and is equal to a positive number $c>0$ on a ball of radius $c'>0$ around the origin. Now we take the square $\phi^{D}=(\phi^{DD})^2$ and see that
		\[
		\hat{\phi}^{D}=\hat{\phi}^{DD}*\hat{\phi}^{DD}.
		\]
		We see that $\hat{\phi}^{D}$ is non-negative, vanishes outside the ball of radius $c''$ around the origin. Unlike $\hat{\phi}^{DD}$, $\hat{\phi}^{D}$ is no longer a constant function on any ball centred at the origin. However there is a number $c'''>0$ such that for each $\omega$ inside $B_{c'''}(0)$, $\hat{\phi}^{D}(\omega)$ is greater than $0.5\hat{\phi}^{D}(0)>0$ and less than $2\hat{\phi}^{D}(0).$ By rescaling, we may assume that $\phi^D(x)\geq 1$ for $x\in B_{2.5}(0).$ This can be done because $\phi^D$ is real valued, Schwartz and $\phi^D(0)>0.$ Since $\hat{\phi}^D$ is compactly supported, we can denote $c''''=\|\hat{\phi}^D\|_\infty.$ We write $h_{g,\delta}=\nu_g*\phi^D(\delta^{-1}.).$ We see that 
		\[
		\nu_g(B_{2.5\delta}(z))=\int_{B_{2.5\delta}(z)}d\nu_g(x)\leq \int \phi^D((z-x)/\delta)d\nu_g(x)=h_{g,\delta}(z).
		\]
	Now we write $f_{g,\delta}(.)=\delta^{-n}h_{g,\delta}(.),$ as a result we see that
		\[
		\int \nu^2_\delta(z)dz\lesssim \int\int f^{k-1}_{g,\delta}(z)d\nu_g(z)dg.
		\]
		Let $\psi$ be a smooth real valued cutoff function supported in $\{\omega\in\mathbb{R}^n: |\omega|\in [0.5,4] \}$ and identically equal to $1$ in $\{\omega\in\mathbb{R}^n: |\omega|\in [1,2] \}$ and bounded from above by $1.$ Let $f_{g,\delta,j}, \nu_{g,j}$ be the $j$-th Littlewood-Paley piece of $f_{g,\delta},\nu_{g}$ respectively, namely, $\hat{f}_{g,\delta,j}(\omega)=\hat{f}_{g,\delta}(\omega)\psi(2^{-j}\omega)$ and similarly for $\nu_{g,j}.$ The rest of the argument is essentially the same as in \cite[Section 3 ]{GILP15}. We need to bound $\|f_{g,\delta,j}\|_{\infty}$ as well as $\|\nu_{g,j}\|_{\infty}.$ The later can be bounded by
		\[
		C2^{j(n-s)}
		\]
		for any $s<\Haus F$ with a constant $C$ depending on the function $\psi$. This was shown in \cite[page 805]{GILP15}. For the former, we will be interested in estimating $\|f_{g,\delta,j}\|_\infty$ when $2^j$ is not as large as $\delta^{-1}.$ In this case, recall that $f_{g,\delta}=\nu_g*\phi^D_\delta$ and in terms of Fourier transform we have
		\[
		\hat{f}_{g,\delta,j}=\hat{\nu}_g \hat{\phi^D_\delta} \psi(2^{-j}.)
		\]
			Recall that $\phi^D_\delta(.)=\delta^{-n}\phi^D(./\delta)$, therefore we have $\hat{\phi^D_\delta}(.)=\hat{\phi^D}(\delta.).$ Then we see that 
		\begin{eqnarray*}
		\|f_{g,\delta,j}\|_\infty&\leq& \|\hat{f}_{g,\delta,j}\|_1\\
		&\leq& c'''' \int |\hat{\nu}_g(\omega)  \psi(2^{-j}\omega)|d\omega\\
		&\leq& c''''\int_{B_{2^{j+2}}(0)} |\hat{\mu}(\omega)\hat{\mu}(g\omega)|d\omega\\
		&\leq& c''''\sqrt{\int_{B_{2^{j+2}}(0)} |\hat{\mu}(\omega)|^2d\omega \int_{B_{2^{j+2}}(0)} |\hat{\mu}(g\omega)|^2d\omega}.
		\end{eqnarray*}
		By the discussion in \cite[Section 3.8]{Ma2} we see that
		\[
		\int_{B_{2^{j+1}}(0)} |\hat{\mu}(\omega)|^2d\omega\lesssim 2^{(j+2)(n-s)}.
		\]
		The same estimate holds for $\int_{B_{2^{j+2}}(0)} |\hat{\mu}(g\omega)|^2d\omega$ as well. Therefore we see that
		\[
			\|f_{g,\delta,j}\|_\infty\leq C'2^{j(n-s)}
		\]
		where $C'>0$ is a constant which does not depend on $g,j,\delta.$ Observe that if $2^{j-1}>c''\delta^{-1}$ then $f_{g,\delta,j}=0$ and this is the reason for considering $2^j$ to be not much larger than $\delta^{-1}.$ By \cite[Formula (3.27)]{Ma2}, we have
		\[
		\int f^{k-1}_{g,\delta}(z)d\nu_g(z)=\int \hat{\nu}_g(\omega)\hat{f}_{g,\delta}*\dots*\hat{f}_{g,\delta}(-\omega)d\omega.\tag{*}
		\]
		We can apply the exactly the same argument in \cite[Section 3]{GILP15}. As a result we see that
		\[
		\int f^{k-1}_{g,\delta}(z)d\nu_g(z)\lesssim \sum_{j}  2^{j(n-s)(k-2)}\int |f_{g,\delta,j}(x)\nu_{g,j}(x)|dx.
		\] 
		By Cauchy-Schwartz we see that
		\[
	\int |f_{g,\delta,j}(x)\nu_{g,j}(x)|dx\leq \sqrt{\int |f_{g,\delta,j}(x)|^2dx\int |\nu_{g,j}(x)|^2dx}.
		\]
		By Plancherel's formula we see that
		\[
		\int |f_{g,\delta,j}(x)|^2dx=\int |\hat{f}_{g,\delta,j}(\omega)|^2d\omega=\int |\hat{\nu}_{g,j}(\omega)|^2 |\hat{\phi}^D_\delta(\omega)|^2d\omega\leq (c'''')^2\int |\hat{\nu}_{g,j}(\omega)|^2d\omega.
		\]
	However, if $2^{j-1}\geq c''\delta^{-1}$ we see that $\hat{\phi}^D_\delta(\omega)=0$ whenever $|\omega|\in [2^{j-1},2^{j+2}].$ Thus in this case we see that
	\[
		\int |f_{g,\delta,j}(x)|^2dx=\int |\hat{\nu}_{g,j}(\omega)|^2 |\hat{\phi}^D_\delta(\omega)|^2d\omega=0.
	\] 
    Therefore we see that for $2^{j-1}\leq c''\delta^{-1},$
	\[
	\int |f_{g,\delta,j}(x)\nu_{g,j}(x)|dx\lesssim \int |\nu_{g,j}(x)|^2dx.
	\]	
	Then by integrating the above inequality against $dg$ we see that
		\[
		\int \nu^2_\delta(z)dz\lesssim \int\int f^{k-1}_{g,\delta}(z)d\nu_g(z)dg\lesssim \sum_{j: 2^{j-1}\leq c''\delta^{-1}} 2^{j(n-s)(k-2)}\int\int |\nu_{g,j}(x)|^2dxdg.
		\]
		Notice hat $\nu_{g,j}$ is real valued. By Plancherel's formula we have
		\[
		\int\int |\nu_{g,j}(x)|^2dxdg=\int \int |\hat{\nu}_{g,j}(\omega)|^2d\omega dg=\int\int |\hat{\nu}_g(\omega)|^2|\psi(2^{-j}\omega)|^2 d\omega dg.
		\]
	Since $\nu_g=\mu*g\mu$ we see that
		\[
		\int\int |\hat{\nu}_g(\omega)|^2|\psi(2^{-j}\omega)|^2 d\omega dg=\int\int |\hat{\mu}(\omega)|^2 |\hat{\mu}(g\omega)|^2 |\psi(2^{-j}\omega)|^2d\omega dg. 
		\]
		Since $\psi$ is radial, for $t\geq 0$ we write $\psi(t)$ for the value $\psi(x)$ with an arbitrary $x$ with norm $|x|=t$. This value is well-defined. Then, up to a multiple constant, the above expression on RHS is equal to
		\[
		\int \left(\int_{S^{n-1}} |\hat{\mu}(t\tau) |^2     d\tau\right)^2 |\psi(2^{-j}t)|^2t^{n-1}dt,\tag{**}
		\]
		where $d\tau$ is the Lebesgue measure on the unit sphere. We see that $(**)$ is bounded from above by
		\[
		\int_{2^{j-1}}^{2^{j+2}} \left(\int_{S^{n-1}} |\hat{\mu}(t\tau) |^2     d\tau\right)^2 t^{n-1}dt.
		\]
		Now we have the following estimate,
		\begin{eqnarray*}
			& &\int f^{k-1}_{g,\delta}(z)d\nu_g(z)dg\\
			&\lesssim&\sum_{j: 2^{j}\leq 2c''\delta^{-1}}   2^{j(n-s)(k-2)}	\int_{2^{j-1}}^{2^{j+2}} \left(\int_{S^{n-1}} |\hat{\mu}(t\tau) |^2     d\tau\right)^2 t^{n-1}dt.
		\end{eqnarray*}
		If $k=3$ then we have (the sum of negative values of $j$ gives a constant)
		\[
		\int f^{2}_{g,\delta}(z)d\nu_g(z)dg\lesssim \sum_{j: 1\leq 2^{j}\leq 2c''\delta^{-1}}2^{j(n-s)}	\int_{2^{j-1}}^{2^{j+2}} \left(\int_{S^{n-1}} |\hat{\mu}(t\tau) |^2     d\tau\right)^2 t^{n-1}dt.
		\]
		Now we are going to use \cite[Theorem 3.1]{GILP15},
		\begin{thm}
			Let $\mu$ be a compactly supported Borel measure on $\mathbb{R}^n$. Then for $s\in (n/2,\Haus \mu),\epsilon>0,$
			\[
			\int_{S^{n-1}} |\hat{\mu}(t\tau)|^2d\tau\lesssim_{\epsilon,s} t^{\epsilon-\gamma_s},
			\]
			where $\gamma_s=(n+2s-2)/4$ if $s\in [n/2,(n+2)/2]$ and $\gamma_s=s-1$ if $s\geq (n+2)/2.$
		\end{thm}
		Here $\Haus \mu$ is defined to be the following value
		\[
		\sup\left\{s>0: \int\int |x-y|^{s}d\mu(x)d\mu(y)<\infty  \right \}.
		\]
		For a compact set $F$ we have $\Haus F=\sup_{\mu\in\mathcal{P}(F)}\Haus \mu.$ In our situation, we recall that $s<\Haus F$ which was mentioned in the beginning of this proof.
		
		If $2^{j-1}\leq c''\delta^{-1}$ we see that
		\[
			\int_{2^{j-1}}^{2^{j+2}} \left(\int_{S^{n-1}} |\hat{\mu}(t\tau) |^2     d\tau\right)^2 t^{n-1}dt\lesssim 2^{j(n-s-\gamma_s)}.
		\]
		This is because (see the beginning of \cite[Section 3.8]{Ma2}),
		\begin{eqnarray*}
		\int_{2^{j-1}}^{2^{j+2}}\left(\int_{S^{n-1}} |\hat{\mu}(t\tau) |^2     d\tau\right)t^{n-1}dt&\lesssim& \int_{B(0,2^{j+2})} |\hat{\mu}(\omega)|^2d\omega	\lesssim 2^{j(n-s)}.
		\end{eqnarray*}
		since $\mu$ is an $s$-Frostman measure. For the sum with small $j,$ we assume that $2n-2s-\gamma_s>0$ for otherwise $\nu$ is absolutely continuous as shown in \cite[Section 3]{GILP15}. We have
		\[
		\sum_{j:1\leq 2^j\leq 2c''\delta^{-1}}2^{j(2n-2s-\gamma_s)}\lesssim_{n,s} \delta^{-(2n-2s-\gamma_s)}.
		\]
		In all we have obtained that for $k=3$ and $2n-2s-\gamma_s>0,$
		\[
		\int f^{2}_{g,\delta}(z)d\nu_g(z)dg\lesssim \delta^{-(2n-2s-\gamma_s)}.
		\]
	Therefore we see that
		\[
		\int \nu^2_\delta(z)dz\lesssim \int f^{2}_{g,\delta}(z)d\nu_g(z)dg\lesssim\delta^{-(2n-2s-\gamma_s)}.
		\]
		The above estimate holds for each $\delta$ with a suitable constant in the symbol $\lesssim$. We claim the following estimate holds,
		\[
		\int_{B_{\delta^{-1}}(0)} |\hat{\nu}(\omega)|^2d\omega\lesssim\delta^{-(2n-2s-\gamma_s)}.
		\]
		To see this, observe that $\hat{\phi}$ is Schwartz, real valued and $\hat{\phi}(0)>0.$ Then there is a positive number $r$ such that $\hat{\phi}(\omega)\geq 0.5\hat{\phi}(0)$ whenever $|\omega|\leq r.$ Then we know that
		\[
		\int_{B_{r\delta^{-1}}(0)} |\hat{\nu}(\omega)|^2d\omega\leq (0.5\hat{\phi}(0))^{-2}\int_{B_{r\delta^{-1}}(0)} |\hat{\nu}(\omega)\hat{\phi}_\delta(\omega)|^2d\omega\leq (0.5\hat{\phi}(0))^{-2}\int |\hat{\nu}(\omega)\hat{\phi}_\delta(\omega)|^2d\omega.
		\]
		Since $\nu_\delta$ is real valued we see that
	    \[
	    \int |\hat{\nu}(\omega)\hat{\phi}_\delta(\omega)|^2d\omega=\int |\hat{\nu_\delta}(\omega)|^2 d\omega=\int \nu^2_\delta(z)dz.
	    \]
		From here the claim follows. Therefore we see that
		\[
		\int |\hat{\nu}(\omega)|^2 |\omega|^{t-3}d\omega\lesssim \int_{B_1(0)} |\omega|^{t-3}d\omega+\sum_{j\geq 0}2^{j(t-3)}\int_{|\omega|\in [2^j,2^{j+1}]} |\hat{\nu}(\omega)|^2d\omega<\infty
		\]
		whenever $0<t<3+2s+\gamma_s-2n.$ This implies that $\nu$ has finite $t$-energy if $t<3+2s+\gamma_s-2n.$ Therefore we see that
		\[
		\Haus \nu\geq 2s+\gamma_s-2n+3.
		\]
		This concludes the proof.
	\end{proof}

	\section{Further comments and problems}\label{Fur}
	A crucial point for the proof of Lemma \ref{lmDCR} is that if we fix a point in $\mathbb{R}^2$ and we want to put $n$ different triangles with the same shape and the corresponding vertex $x$ then the other two vertices trace out a subset of two concentric circles. On each of these circles, we have at least $n/2$ points. So we get $n^2/4$ many different pairs for incidence counting. This does not hold in $\mathbb{R}^3$. For example, take a $3$-simplex and put $n$ many rotated copies of this simplex around a fixed point in $\mathbb{R}^3$. Then the other $3$ points trace out three (not necessarily different) concentric spheres. However, it can happen that one of the spheres contains only $1$ point. So we can not hope to get roughly $n^3$ many triples for the incidence. Instead, it is possible to show that the number of triples is at least roughly $n^2$. The situation is more complicated for higher dimensions. We think that the following general result holds.
	
	\begin{conj}\label{con}
		Let $F\subset\mathbb{R}^n,n\geq 1$ be a compact set. Then the following result holds
		\[
		\Assouad \Delta_{n}(F)\geq \frac{n}{2}\Assouad F.
		\]
	\end{conj}    
	More generally we can ask the following question
	\begin{ques}
		Let $F\subset\mathbb{R}^n$ with $n\geq 1$, then for an integer $k\in [2,n+1]$ what is the relation between $\Assouad F$ and $\Delta_k(F)$?
	\end{ques}
	In particular, when $k=2$ we encounter the distance set problem and in this case we have the following result by \cite[Theorem 2.9]{FHY}.
	\begin{thm*}
		Let $F\subset\mathbb{R}^n$ with $n\geq 1$, we have the following result
		\[
		\Assouad D(F)\geq \frac{1}{d}\Assouad F.
		\]
		The inequality is strict unless $\Assouad D(F)=1$.
	\end{thm*}

	\section{Acknowledgement}
	HY was financially supported by the University of St Andrews.

	\providecommand{\bysame}{\leavevmode\hbox to3em{\hrulefill}\thinspace}
	\providecommand{\MR}{\relax\ifhmode\unskip\space\fi MR }
	\providecommand{\MRhref}[2]{%
		\href{http://www.ams.org/mathscinet-getitem?mr=#1}{#2}
	}
	\providecommand{\href}[2]{#2}
	
	\bibliographystyle{amsalpha}
	
\end{document}